\documentclass[a4,12pt]{article}
\usepackage{amsmath}
\usepackage{amssymb}
\usepackage{amsthm}
\newtheorem{lemma}{Lemma}
\newtheorem{corollary}{Corollary}
\newtheorem{theorem}{Theorem}
\theoremstyle{definition}
\newtheorem*{corrections}{Corrections to [12]}
\newtheorem*{note}{Note}
\newtheorem*{problems}{Problems}

\begin{document}
\title{Note on cyclic sum of certain parametrized multiple series}
\author{Masahiro Igarashi}
\date{}
\maketitle
\begin{abstract} 
We prove the cyclic sum formulas for certain two-parameter multiple series. These are new and non-trivial generalizations of the cyclic sum formulas for multiple zeta values and multiple zeta-star values. 
\end{abstract}
\begin{flushleft}
\end{flushleft}
\begin{flushleft}
\end{flushleft}
\section{Introduction}
In the present paper, we deal with relations for our two-parameter multiple series 
\begin{equation}
\sum_{0\le m_1< \cdots<m_{n}<\infty}
\frac{(\alpha)_{m_1}}{{m_1}!}
\frac{{m_n}!}{(\alpha)_{m_n}}
\left\{\prod_{i=1}^{n}\frac{1}{(m_i+\alpha)^{a_i}(m_i+\beta)^{b_i}}\right\},
\end{equation}
\begin{equation}
\sum_{0\le m_1{\le} \cdots{\le}m_{n}<\infty}
\frac{(\alpha)_{m_1}}{{m_1}!}
\frac{{m_n}!}{(\alpha)_{m_n}}
\left\{\prod_{i=1}^{n}\frac{1}{(m_i+\alpha)^{a_i}(m_i+\beta)^{b_i}}\right\},
\end{equation}
where $n\in\mathbb{Z}_{\ge1}$; $a_i,b_i\in\mathbb{Z}$ ($i=1,\ldots,n$) such that $a_i+b_i\ge1$ ($i=1,\ldots,n-1$), $a_n+b_n\ge2$; $\alpha,\beta\in\mathbb{C}$ such that $\mathrm{Re}(\alpha)>0$, $\beta\notin\mathbb{Z}_{\le0}$. 
Here $\mathbb{Z}_{\ge{k}}:=\{k,k+1,k+2,\ldots\}$ and 
$\mathbb{Z}_{\le{k}}:=\{k,k-1,k-2,\ldots\}$ for $k\in\mathbb{Z}$. The symbol $(a)_m$ denotes the Pochhammer symbol, i.e., $(a)_m = a(a+1)\cdots(a+m-1)$ ($m\in\mathbb{Z}_{\ge1}$) and $(a)_0 =1$. 
The multiple series (1) and (2) are a non-trivial two-parameter extension of the multiple zeta value $\zeta(\{k_i\}^{n}_{i=1})$ (MZV for short) and of the multiple zeta-star value $\zeta^{\star}(\{k_i\}^{n}_{i=1})$ (MZSV for short), 
respectively, where 
\begin{equation*}
\begin{aligned}
&\zeta(\{k_i\}^{n}_{i=1}) := \sum_{0< m_1< \cdots<m_n<\infty}\frac{1}{m_1^{k_1} \cdots m_n^{k_n}},\\ 
&\zeta^{\star}(\{k_i\}^{n}_{i=1}) := \sum_{1\le m_1\le \cdots{\le}m_n<\infty}\frac{1}{m_1^{k_1} \cdots m_n^{k_n}}
\end{aligned}
\end{equation*}
($n\in\mathbb{Z}_{\ge1}$; $k_i\in\mathbb{Z}_{\ge1}$ ($i=1,\ldots,n-1$), 
$k_n\in\mathbb{Z}_{\ge2}$; $\{k_i\}^{n}_{i=1}:=k_1,\ldots,k_n$); see Euler \cite{eu}, Hoffman \cite{h}, Zagier \cite{z}, 
which are the pioneer works on MZV and MZSV. The study of this kind of extension was originated by the author in \cite{i2007}. 
One of the interesting properties of MZV and MZSV is that they satisfy 
various interesting relations. 
Our extensions keep this property. For example, they have the sum and the duality formula, their generalizations 
(see \cite{i2007} and \cite{i2009}) and non-trivial evaluations derived 
from a hypergeometric identity (see \cite{i4}). 
An advantage of (1) and (2) is that they have a good derivation property on the parameters $\alpha$ and $\beta$, that is, their partial derivatives can be 
expressed in $\mathbb{Z}$-linear combinations of themselves; see \textbf{(C1)} of Section 3 and the proof of \cite[Lemma 2.5]{i2009}. 
This allows us to derive many relations from one relation. In Section 3, we shall give such an application of my results. In the present paper, we prove 
a new class of relations among (1) and (2), the cyclic sum formula (CSF for short). This shows that our multiple series are the right extensions of MZV 
and MZSV. Here we put
\begin{equation*}
\begin{aligned}
&Z_{I}(\{k_i\}^{n}_{i=1};(\alpha,\beta))\\
:=
&\sum_{0\le m_1< \cdots<m_{n}<\infty}
\frac{(\alpha)_{m_1}}{{m_1}!}
\frac{{m_n}!}{(\alpha)_{m_n}}
\frac{1}{(m_1+\beta)^{k_1}}\left\{\prod_{i=2}^{n}\frac{1}{(m_i+\alpha)(m_i+\beta)^{k_i-1}}\right\},\\
&Z_{II}(\{k_i\}^{n}_{i=1};(\alpha,\beta))\\
:=
&\sum_{0\le m_1< \cdots<m_{n}<\infty}
\frac{(\alpha)_{m_1}}{{m_1}!}
\frac{{m_n}!}{(\alpha)_{m_n}}
\left\{\prod_{i=1}^{n-1}\frac{1}{(m_i+\alpha)(m_i+\beta)^{k_i-1}}\right\}\\
&\times\frac{1}{(m_n+\alpha)^2(m_n+\beta)^{k_n-2}},\\
&Z^{\star}_{I}(\{k_i\}^{n}_{i=1};(\alpha, \beta))\\
:=
&\sum_{0\le m_1\le \cdots {\le}m_n<\infty}\frac{(\alpha)_{m_1}}{{m_1}!}\frac{{m_n}!}{(\alpha)_{m_{n}}}
\frac{1}{(m_1+\beta)^{k_1}}
\left\{ \prod_{i=2}^{n}\frac{1}{(m_i+\alpha)(m_i+\beta)^{k_i-1}} \right\},\\
&Z(a|b;(\alpha^{'},\beta)):=\sum_{m=0}^{\infty}
\frac{1}{(m+\alpha^{'})^{a}(m+\beta)^{b}},
\end{aligned}
\end{equation*}
where $k_i\in\mathbb{Z}_{\ge1}$ ($i=1,\ldots,n-1$), $k_n\in\mathbb{Z}_{\ge2}$; 
$a,b\in\mathbb{Z}$ such that $a+b\ge2$; $\alpha,\alpha^{'},\beta\in\mathbb{C}$ such that $\mathrm{Re}(\alpha)>0$ and $\alpha^{'},
\beta\notin\mathbb{Z}_{\le0}$. The last single series is the case $n=1$ of both (1) and (2). Our main results are as follows: 
\begin{theorem}
Let $k_1,\ldots,k_n\in\mathbb{Z}_{\ge1}$ such that $k_i\ge2$ for some $i$ $(i=1,\ldots,n)$. 
Then 
\par
$(i) (\textrm{April 2013})$
\begin{equation}
\begin{aligned}
&\sum_{i=1}^{n}\sum_{j=0}^{k_i-2}Z_{I}(j+1,k_{i+1},\ldots,k_n,k_1,\ldots,k_{i-1},k_i-j;(\alpha,\beta))\\
=
&\sum_{i=1}^{n}Z_{II}(k_{i+1},\ldots,k_n,k_1,\ldots,k_{i-1},k_i+1;(\alpha,\beta)),
\end{aligned}
\end{equation}
\par
$(ii) (\textrm{October 2012})$
\begin{equation}
\begin{aligned}
&\sum_{i=1}^{n}\sum_{j=0}^{k_i-2}Z^{\star}_{I}(j+1,k_{i+1},\ldots,k_n,k_1,\ldots,k_{i-1},k_i-j;(\alpha,\beta))\\
=
&(k-n)Z(n|k-n+1;(\alpha,\beta))+nZ(n+1|k-n;(\alpha,\beta))
\end{aligned}
\end{equation}
for all $\alpha,\beta\in\mathbb{C}$ such that $\mathrm{Re}(\alpha)>0$, $\beta\notin\mathbb{Z}_{\le0}$, where $k:=k_1+\cdots+k_n$. 
Here the empty sum $\sum_{j=0}^{-1}$ is regarded as $0$. 
\end{theorem}
These relations are new and non-trivial generalizations of CSFs for MZVs (Hoffman and Ohno \cite[p.~333]{ho}) and MZSVs (Ohno and Wakabayashi \cite[Theorem 1]{ow}): the former is the case $\alpha=\beta=1$ of (3) and 
the latter is that of (4). 
Another simple specialization $\alpha=\beta$ of (3) and of (4) give 
CSFs for one-parameter multiple series proved in 
\cite[Theorem 1.1 (i) and (ii)]{i}; therefore Theorem 1 is a two-parameter extension of the theorem. 
The relation (4) yields also a new sum formula for (2); see (9) below. 
It can be seen that the multiple series (1) and (2) have a flavor of 
multiple Hurwitz zeta values (MHZVs for short). Our results give also relations for MHZVs. 
In fact, the case $\alpha=1$ of (3), (4) and of (9) give the cyclic sum and the sum formula for the MHZVs 
\begin{equation*}
\sum_{0\le m_1{\lesssim}\cdots{\lesssim}m_{n}<\infty}
\prod_{i=1}^{n}\frac{1}{(m_i+1)^{a_i}(m_i+\beta)^{b_i}},
\end{equation*}
where the symbol $\lesssim$ denotes $<$ or $\le$. For other relations for 
MHZVs, see \cite{i2007}, \cite{i4}, \cite[Note 2]{i2020}. For CSFs and related topics ($q$-analogues etc.), 
see also \cite{ikoo}, \cite{o1}, \cite{o2}, \cite{tw}, \cite{b}, \cite{oo}, \cite{ooz}, \cite{ktw}. We shall prove Theorem 1 in Section 2. In Section 3, 
we shall give applications of Theorem 1 and also pose some problems. The contents of the present paper are results in the author's research project on parametrized multiple series. See also \cite[Note 2 (iv) and (v)]{i2020}. 
The identities (4), (9) and the identity (3) have been known to be 
results of the author since October 2012 and April 2013, respectively; 
see the following Note. 
\begin{note}[2022] 
The present paper is a revised version of manuscripts of mine which were 
submitted to a journal on March 15, 2013 and February 17, 2016. 
(My manuscript submitted on March 15, 2013 was distributed as a preprint on December 5, 2013.) I proved (4) and (9) in October 2012, and wrote them in 
an unpublished manuscript in October 2012. 
I wrote about (4) also in a postdoctoral research plan document submitted to a previous affiliation of mine (Graduate School of Mathematics, Nagoya University, Japan) in February 2013. 
As regards (3), I proved it in April 2013, and wrote it in an unpublished manuscript on April 9, 2013. 
As with (4), I wrote about (3) also in a postdoctoral research plan document 
submitted to my previous affiliation in February 2014. 
All of the identities (3), (4) and (9) were written in a preprint of 
mine distributed on February 12, 2015 and in a manuscript of mine submitted 
to a journal on February 6, 2018 (rejected on June 11, 2018). 
As regards Section 3, the contents of \textbf{(C2)} 
and of the former part of \textbf{(C1)} (i.e., the part related to (4)) were written in my manuscript submitted on March 15, 2013 and in its preprint distributed on December 5, 2013. 
The contents of the latter part of \textbf{(C1)} (i.e., the part related to (3)) were written in 
my preprint distributed on February 12, 2015 and in my manuscript submitted 
on February 6, 2018. The contents of \textbf{(C3)} were written in my manuscript submitted on February 17, 2016 (rejected on March 10, 2016). 
A problem on the multiple series (18) posed in \textbf{(C3)} was written 
also in a postdoctoral research plan document of mine submitted to my previous affiliation in February 2016. 
I submitted the present version to many journals in 
February 2018--February 2022. In this Note and other places of the present paper, I referred to research documents of mine of October 2012--February 2018: unpublished manuscripts, submissions 
to journals, preprints and postdoctoral research plan documents. I note that all of those research documents are electronic files made in a computer system of 
my previous affiliation in October 2012--February 2018. 
(This can be confirmed by asking people who read the contents of 
my electronic files in 2012--2018.) In fact, in 2012--March 2018, 
I used the computer system whenever I wrote research documents. 
\end{note}
\section{Proof of Theorem 1}
Hoffman and Ohno \cite{ho}, Ohno and Wakabayashi \cite{ow} proved 
CSFs for MZVs and MZSVs by using certain auxiliary multiple series. 
To prove Theorem 1, we use their method. 
We introduce the following auxiliary multiple series:
\begin{equation*}
\begin{aligned}
&T(\{k_i\}^{n}_{i=1};(\alpha,\beta))\\
:=&\sum_{\begin{subarray}{c}
          0\le m_0<m_1< \cdots<m_n<\infty
         \end{subarray}}
\frac{(\alpha)_{m_0}}{{m_0}!}\frac{{m_n}!}{(\alpha)_{m_{n}}}
\left\{ \prod_{i=1}^{n}\frac{1}{(m_i+\alpha)(m_i+\beta)^{k_i-1}} \right\}
\frac{1}{m_n-m_0},
\end{aligned}
\end{equation*}
\begin{equation*}
\begin{aligned}
&T^{\star}(\{k_i\}^{n}_{i=1};(\alpha,\beta))\\
:=&\sum_{\begin{subarray}{c}
          0\le m_0{\le}m_1{\le} \cdots{\le}m_n<\infty\\
          m_0{\neq}m_n
         \end{subarray}}
\frac{(\alpha)_{m_0}}{{m_0}!}\frac{{m_n}!}{(\alpha)_{m_{n}}}
\left\{ \prod_{i=1}^{n}\frac{1}{(m_i+\alpha)(m_i+\beta)^{k_i-1}} \right\}
\frac{1}{m_n-m_0}.
\end{aligned}
\end{equation*}
These multiple series converge absolutely for 
$k_1,\ldots,k_n\in\mathbb{Z}_{\ge1}$ such that $k_i\ge2$ for some $i$ $($$i=1,\ldots,n$$)$ 
and $\alpha,\beta\in\mathbb{C}$ such that $\mathrm{Re}(\alpha)>0$, 
$\beta\notin\mathbb{Z}_{\le0}$. For the proof, refer to \cite[Proof of Lemma 2.1]{i}. The key to proving CSF in \cite{ho} and \cite{ow} was 
a cyclic property of auxiliary multiple series (see \cite[Theorem 3.2]{ho} and 
\cite[Key Lemma 1]{ow}). The above multiple series also have the property, 
that is, they satisfy the following identities (see also \cite[Lemma 2.5]{i}): 
\begin{lemma}
Let $k_1,\ldots,k_n\in\mathbb{Z}_{\ge1}$ such that $k_i\ge2$ for some $i$ $(i=1,\ldots,n)$. 
Then 
\par
$(i)$
\begin{equation}
\begin{aligned}
&T(k_1,\ldots,k_n;(\alpha,\beta))-T(k_n,k_1,\ldots,k_{n-1};(\alpha,\beta))\\
=&Z_{II}(k_n,k_1,\ldots,k_{n-2},k_{n-1}+1;(\alpha,\beta))\\
&-\sum_{j=0}^{k_n-2}Z_{I}(j+1,k_1,\ldots,k_{n-1},k_n-j;(\alpha,\beta)),
\end{aligned}
\end{equation}
\par
$(ii)$
\begin{equation}
\begin{aligned}
&T^{\star}(k_1,\ldots,k_n;(\alpha,\beta))-T^{\star}(k_n,k_1,\ldots,k_{n-1};(\alpha,\beta))\\
=&(k_n-1)Z(n|k-n+1;(\alpha,\beta))+Z(n+1|k-n;(\alpha,\beta))\\
&-\sum_{j=0}^{k_n-2}Z^{\star}_{I}(j+1,k_1,\ldots,k_{n-1},k_n-j;(\alpha,\beta))
\end{aligned}
\end{equation}
for all $\alpha,\beta\in\mathbb{C}$ such that $\mathrm{Re}(\alpha)>0$, $\beta\notin\mathbb{Z}_{\le0}$, where $k:=k_1+\cdots+k_n$. 
Here the empty sum $\sum_{j=0}^{-1}$ is regarded as $0$. 
\end{lemma}
\begin{proof}
We prove (6) in the same way as in \cite[Proof of Key Lemma 1]{ow} with 
\cite[Remark 2.4]{i}. By using a partial fraction decomposition, we have 
\begin{equation*}
\begin{aligned}
&\sum_{\begin{subarray}{c}
          0\le m_0{\le}m_1{\le} \cdots{\le}m_n<\infty\\
          m_0{\neq}m_n
         \end{subarray}}
\frac{(\alpha)_{m_0}}{{m_0}!}\frac{{m_n}!}{(\alpha)_{m_{n}}}
\frac{1}{(m_0+\beta)^j}
\left\{ \prod_{i=1}^{n-1}\frac{1}{(m_i+\alpha)(m_i+\beta)^{k_i-1}} \right\}\\
&\hspace*{1.5cm}\times\frac{1}{(m_n+\alpha)(m_n+\beta)^{k_n-1-j}(m_n-m_0)}\\
=
&\sum_{\begin{subarray}{c}
          0\le m_0{\le}m_1{\le} \cdots{\le}m_n<\infty\\
          m_0{\neq}m_n
         \end{subarray}}
\frac{(\alpha)_{m_0}}{{m_0}!}\frac{{m_n}!}{(\alpha)_{m_{n}}}
\frac{1}{(m_0+\beta)^{j+1}}
\left\{ \prod_{i=1}^{n-1}\frac{1}{(m_i+\alpha)(m_i+\beta)^{k_i-1}} \right\}\\
&\hspace*{1.5cm}\times\frac{1}{(m_n+\alpha)(m_n+\beta)^{k_n-2-j}}
\left(\frac{1}{m_n-m_0}-\frac{1}{m_n+\beta}\right)\\
=
&\sum_{\begin{subarray}{c}
          0\le m_0{\le}m_1{\le} \cdots{\le}m_n<\infty\\
          m_0{\neq}m_n
         \end{subarray}}
\frac{(\alpha)_{m_0}}{{m_0}!}\frac{{m_n}!}{(\alpha)_{m_{n}}}
\frac{1}{(m_0+\beta)^{j+1}}
\left\{ \prod_{i=1}^{n-1}\frac{1}{(m_i+\alpha)(m_i+\beta)^{k_i-1}} \right\}\\
&\hspace*{1.5cm}\times\frac{1}{(m_n+\alpha)(m_n+\beta)^{k_n-2-j}(m_n-m_0)}\\
&-Z^{\star}_{I}(j+1,k_1,\ldots,k_{n-1},k_n-j;(\alpha,\beta))\\
&+Z(n|k_1+\cdots+k_n-n+1;(\alpha,\beta))
\end{aligned}
\end{equation*}
for $j=0, \ldots, k_n-2$. Further, taking the sum $\sum_{j=0}^{k_n-2}$ 
on both sides, we have 
\begin{equation}
\begin{aligned}
&T^{\star}(k_1,\ldots,k_n;(\alpha,\beta))\\
=
&\sum_{\begin{subarray}{c}
          0\le m_0{\le}m_1{\le} \cdots{\le}m_n<\infty\\
          m_0{\neq}m_n
         \end{subarray}}
\frac{(\alpha)_{m_0}}{{m_0}!}\frac{{m_n}!}{(\alpha)_{m_{n}}}
\frac{1}{(m_0+\beta)^{k_n-1}}
\left\{ \prod_{i=1}^{n-1}\frac{1}{(m_i+\alpha)(m_i+\beta)^{k_i-1}} \right\}\\
&\hspace*{1.5cm}\times\frac{1}{(m_n+\alpha)(m_n-m_0)}\\
&-\sum_{j=0}^{k_n-2}Z^{\star}_{I}(j+1,k_1,\ldots,k_{n-1},k_n-j;(\alpha,\beta)) \\
&+(k_n-1)Z(n|k_1+\cdots+k_n-n+1;(\alpha,\beta)).
\end{aligned}
\end{equation}
Here we rewrite the first term on the right-hand side of (7) as 
\begin{equation}
\begin{aligned}
&\sum_{\begin{subarray}{c}
          0\le m_0{\le}m_1{\le} \cdots{\le}m_n<\infty\\
          m_0{\neq}m_n
         \end{subarray}}
\frac{(\alpha)_{m_0}}{{m_0}!}\frac{{m_n}!}{(\alpha)_{m_{n}}}
\frac{1}{(m_0+\beta)^{k_n-1}}
\left\{ \prod_{i=1}^{n-1}\frac{1}{(m_i+\alpha)(m_i+\beta)^{k_i-1}} \right\}\\
&\hspace*{1.5cm}\times\frac{1}{(m_n+\alpha)(m_n-m_0)}\\
=
&\sum_{\begin{subarray}{c}
          0\le m_0{\le}m_1{\le} \cdots{\le}m_{n-1}<\infty\\
          m_0{\neq}m_{n-1}
         \end{subarray}}
\frac{1}{(m_0+\alpha)(m_0+\beta)^{k_n-1}}
\left\{ \prod_{i=1}^{n-1}\frac{1}{(m_i+\alpha)(m_i+\beta)^{k_i-1}} \right\}\\
&\hspace*{1.5cm}
\times\frac{(\alpha)_{m_0+1}}{{m_0}!}\sum_{m_n=m_{n-1}}^{\infty}\frac{{m_n}!}{(\alpha)_{m_{n}+1}}\frac{1}{m_n-m_0}\\
&+
\sum_{\begin{subarray}{c}
          0\le m_0=\cdots=m_{n-1}<\infty
          \end{subarray}}
\frac{1}{(m_0+\alpha)(m_0+\beta)^{k_n-1}}
\left\{ \prod_{i=1}^{n-1}\frac{1}{(m_i+\alpha)(m_i+\beta)^{k_i-1}} \right\}\\
&\hspace*{1.5cm}
\times\frac{(\alpha)_{m_0+1}}{{m_0}!}\sum_{m_n=m_{n-1}+1}^{\infty}\frac{{m_n}!}{(\alpha)_{m_{n}+1}}\frac{1}{m_n-m_0}.
\end{aligned}
\end{equation}
Using 
\begin{equation*}
\frac{(\alpha)_{m+1}}{m!}\sum_{l=n}^{\infty}\frac{l!}{(\alpha)_{l+1}}\frac{1}{l-m}
= \frac{n!}{(\alpha)_{n}}\sum_{l=0}^{m}\frac{(\alpha)_{l}}{l!}\frac{1}{n-l}
\end{equation*}
($m,n\in\mathbb{Z}$ such that $0{\le}m<n$ and $\alpha\in\mathbb{C}$ with $\mathrm{Re}(\alpha)>0$; \cite[Remark 2.4]{i}) to the inner sums, we have 
\begin{equation*}
\begin{aligned}
(8)
=
&\sum_{\begin{subarray}{c}
          0\le m_0{\le}m_1{\le} \cdots{\le}m_{n-1}<\infty\\
          m_0{\neq}m_{n-1}
         \end{subarray}}
\frac{1}{(m_0+\alpha)(m_0+\beta)^{k_n-1}}
\left\{ \prod_{i=1}^{n-1}\frac{1}{(m_i+\alpha)(m_i+\beta)^{k_i-1}} \right\}\\
&\hspace*{1.5cm}
\times\frac{{m_{n-1}}!}{(\alpha)_{m_{n-1}}}\sum_{l=0}^{m_0}\frac{(\alpha)_l}{l!}\frac{1}{m_{n-1}-l}\\
&+
\sum_{\begin{subarray}{c}
          0\le m_0=\cdots=m_{n-1}<\infty
         \end{subarray}}
\frac{1}{(m_0+\alpha)(m_0+\beta)^{k_n-1}}
\left\{ \prod_{i=1}^{n-1}\frac{1}{(m_i+\alpha)(m_i+\beta)^{k_i-1}} \right\}\\
&\hspace*{1.5cm}
\times
\left\{\frac{{m_{n-1}}!}{(\alpha)_{m_{n-1}}}\sum_{l=0}^{m_0-1}\frac{(\alpha)_l}{l!}\frac{1}{m_{n-1}-l}
+\frac{(\alpha)_{m_0}}{{m_0}!}\frac{{m_{n-1}}!}{(\alpha)_{m_{n-1}+1}}\right\}\\
=
&\sum_{\begin{subarray}{c}
          0\le l\le m_0{\le}m_1{\le} \cdots{\le}m_{n-1}<\infty\\
          l{\neq}m_{n-1}
         \end{subarray}}
\frac{(\alpha)_l}{l!}\frac{{m_{n-1}}!}{(\alpha)_{m_{n-1}}}
\frac{1}{(m_0+\alpha)(m_0+\beta)^{k_n-1}}\\
&\hspace*{1.5cm}
\times\left\{ \prod_{i=1}^{n-1}\frac{1}{(m_i+\alpha)(m_i+\beta)^{k_i-1}} \right\}\frac{1}{m_{n-1}-l}\\
&+
\sum_{\begin{subarray}{c}
          0\le m_0=\cdots=m_{n-1}<\infty
         \end{subarray}}
\frac{1}{(m_0+\alpha)(m_0+\beta)^{k_n-1}}
\left\{ \prod_{i=1}^{n-1}\frac{1}{(m_i+\alpha)(m_i+\beta)^{k_i-1}} \right\}\\
&\hspace*{1.5cm}
\times\frac{(\alpha)_{m_0}}{{m_0}!}\frac{{m_{n-1}}!}{(\alpha)_{m_{n-1}+1}}\\
=
&T^{\star}(k_n,k_1,\ldots,k_{n-1};(\alpha,\beta))
+Z(n+1|k_1+\cdots+k_n-n;(\alpha,\beta)).
\end{aligned}
\end{equation*}
Hence, substituting this into the right-hand side of (7), we obtain (6). 
The identity (5) can be proved in the same way as in \cite[Section 3]{ho} (see also \cite{o1} and \cite{o2}) with \cite[Lemma 2.2]{i}. For more details, refer to \cite[Proof of Lemma 2.5]{i}.
\end{proof}
\begin{proof}[Proof of Theorem {$1$}]
Applying Lemma 1 (i) and (ii) to the index $(k_{i+1},\ldots,k_n,k_1,\ldots,k_i)$ and 
taking the sum $\sum_{i=1}^{n}$ on both sides of the results, we obtain (3) and (4), respectively.
\end{proof}
Using Theorem 1 (ii) and Ohno's argument in \cite[p.~4]{o1}, \cite[pp.~138--139]{o2} (see also \cite[Corollary 2.4 and its proof]{ho}, \cite[p.~294]{ow}, 
\cite[pp.~514--516]{i}), we can prove the following relation:
\begin{corollary}[October 2012]
We have 
\begin{equation}
\begin{aligned}
&\sum_{\begin{subarray}{l}k_1+\cdots+k_n= k\\
       k_i\in\mathbb{Z}_{\ge1} (i=1,\ldots,n-1),\\
k_n\in\mathbb{Z}_{\ge2}
      \end{subarray}}
Z^{\star}_{I}(\{k_i\}_{i=1}^{n};(\alpha,\beta))\\
=
&\binom{k-2}{n-1}Z(n-1|k-n+1;(\alpha,\beta))+\binom{k-2}{n-2}Z(n|k-n;(\alpha,\beta))
\end{aligned}
\end{equation}
for all $k,n\in\mathbb{Z}$ such that $0<n<k$ and $\alpha,\beta\in\mathbb{C}$ such that $\mathrm{Re}(\alpha)>0$, $\beta\notin\mathbb{Z}_{\le0}$.
\end{corollary}
Our relation (9) is a new sum formula for (2); compare (9) with 
the sum formulas of \cite[\textbf{(R2)} of Section 3]{i4}, \cite{i2007}, 
\cite[Remark 2.4]{i2009}. 
Taking $\alpha=\beta=1$ and $\alpha=\beta$ in (9), we have 
the sum formula for MZSVs (Granville \cite{gr}, Hoffman \cite{h}, Zagier) and 
for one-parameter multiple series proved in \cite[Corollary 2.8 (ii)]{i}, respectively.
\section{Applications and problems}
We give applications of Theorem 1 and also pose some problems. Hereafter we use the following notation: 
\begin{equation*}
\begin{aligned}
&\{k_1,\ldots,k_m\}^{n}:=\underbrace{k_1,\ldots,k_m,\ldots,k_1,\ldots,k_m}_{mn},\\
&\{a_{1i},\ldots,a_{mi}\}_{i=1}^{n}:=a_{11},\ldots,a_{m1},\ldots,a_{1n},\ldots,a_{mn}
\end{aligned}
\end{equation*}
($m,n\in\mathbb{Z}_{\ge1}$); we regard $\{k_1,\ldots,k_m\}^{0}$ and $\{a_{1i},\ldots,a_{mi}\}_{i=1}^{0}$ as the empty set $\emptyset$.
\par
$\textbf{(C1)}$ 
By partial differentiation, we derive relations among (1) and (2) from 
the relations (3) and (4). We use some identities for the Pochhammer symbol 
$(a)_{m}$ to 
calculate the partial derivatives on $\alpha$. They fairly simplify the calculation.  Using 
\begin{equation*}
\frac{(\alpha)_{m_1}}{(\alpha)_{m_n}}\prod_{i=2}^{n}\frac{1}{m_i+\alpha}
=
\prod_{i=2}^{n}\frac{(\alpha)_{m_{i-1}}}{(\alpha)_{m_i+1}},
\end{equation*}
we have 
\begin{equation}
\begin{aligned}
&\frac{(-1)^r}{r!}\frac{\mathrm{d}^r}{\mathrm{d}\alpha^r}
\left(\frac{(\alpha)_{m_1}}{(\alpha)_{m_n}}\prod_{i=2}^{n}\frac{1}{m_i+\alpha}
\right)\\
=&
\sum_{\begin{subarray}{c}
r_1+\cdots+r_{n-1}= r\\
       r_i\in\mathbb{Z}_{\ge0}
      \end{subarray}}
\prod_{i=2}^{n}\frac{(\alpha)_{m_{i-1}}}{(\alpha)_{m_i+1}}
\sum_{m_{i-1}\le{l_1}\le\cdots\le{l_{r_{i-1}}}\le{m_i}}\frac{1}{(l_1+\alpha)\cdots(l_{r_{i-1}}+\alpha)}\\
=&
\frac{(\alpha)_{m_1}}{(\alpha)_{m_n}}\left(\prod_{i=2}^{n}\frac{1}{m_i+\alpha}\right)\\
&\times
\sum_{\begin{subarray}{c}
r_1+\cdots+r_{n-1}= r\\
       r_i\in\mathbb{Z}_{\ge0}
      \end{subarray}}
\sum_{\begin{subarray}{c}
m_1{\le}m_{11}\le\cdots{\le}m_{1 r_1}{\le}m_2\\
\vdots\\
m_{n-1}{\le}m_{n-1 1}\le\cdots{\le}m_{n-1 r_{n-1}}{\le}m_n
\end{subarray}}
\prod_{i=1}^{n-1}\prod_{j=1}^{r_i}\frac{1}{m_{ij}+\alpha}
\end{aligned}
\end{equation}
for $n\in\mathbb{Z}_{\ge1}$, $r\in\mathbb{Z}_{\ge0}$ and $m_1,\ldots,m_n\in\mathbb{Z}$ such that $0\le{m_1}\le\cdots{\le}m_n$. 
These identities give the following expression: 
\begin{equation}
\begin{aligned}
\frac{(-1)^r}{r!}\frac{\mathrm{\partial}^r}{\mathrm{\partial}\alpha^r}
Z^{\star}_{I}(\{k_i\}_{i=1}^{n};(\alpha,\beta))
=
\sum_{\begin{subarray}{c}
r_1+\cdots+r_{n-1}= r\\
       r_i\in\mathbb{Z}_{\ge0}
      \end{subarray}}
Z^{\star}_{I}(\{k_i,\{1\}^{r_i}\}_{i=1}^{n-1},k_n;(\alpha,\beta))
\end{aligned}
\end{equation}
for $r\in\mathbb{Z}_{\ge0}$. It is obvious that the partial derivative 
$\frac{(-1)^r}{r!}\frac{\mathrm{\partial}^r}{\mathrm{\partial}\beta^r}Z^{\star}_{I}(\{k_i\}_{i=1}^{n};(\alpha,\beta))$ 
can also be expressed in a $\mathbb{Z}$-linear combination of $Z^{\star}_{I}(\{k_i\}_{i=1}^{n};(\alpha,\beta))$. 
Therefore differentiating both sides of (4) gives further relations among (2). 
In addition, the above examination shows that the partial differentiation on 
$\alpha$ and on $\beta$ give different relations. For instance, we can obtain the following: For the index $(\{\{1\}^{m-1},2\}^{n})$ ($m,n\in\mathbb{Z}_{\ge1}$), 
the relation (4) becomes  
\begin{equation}
\begin{aligned}
Z^{\star}_{I}(1,\{\{1\}^{m-1},2\}^{n};(\alpha,\beta))
=
Z(mn|n+1;(\alpha,\beta))+mZ(mn+1|n;(\alpha,\beta))
\end{aligned}
\end{equation}
($m,n\in\mathbb{Z}_{\ge1}$, $\alpha,\beta\in\mathbb{C}$ such that 
$\mathrm{Re}(\alpha)>0$, $\beta\notin\mathbb{Z}_{\le0}$). 
Differentiating both sides of (12) $r$ times on $\alpha$ and using (11), 
we have the relation 
\begin{equation*}
\begin{aligned}
&\sum_{\begin{subarray}{c}
r_1+\cdots+r_n= r\\
       r_i\in\mathbb{Z}_{\ge0}
      \end{subarray}}
Z^{\star}_{I}(1,\{\{1\}^{m+r_i-1},2\}_{i=1}^{n};(\alpha,\beta))\\
=
&\binom{mn+r-1}{r}Z(mn+r|n+1;(\alpha,\beta))
+m\binom{mn+r}{r}Z(mn+r+1|n;(\alpha,\beta))
\end{aligned}
\end{equation*}
for $r\in\mathbb{Z}_{\ge0}$. 
On the other hand, differentiating both sides of (12) $r$ times on $\beta$, 
we have the relation 
\begin{equation*}
\begin{aligned}
&\sum_{\begin{subarray}{c}
r_1+\cdots+r_{n+1}= r\\
       r_i\in\mathbb{Z}_{\ge0}
      \end{subarray}}
Z^{\star}_{I}(1+r_1,\{\{1\}^{m-1},2+r_i\}_{i=2}^{n+1};(\alpha,\beta))\\
=
&\binom{n+r}{r}Z(mn|n+r+1;(\alpha,\beta))
+m\binom{n+r-1}{r}Z(mn+1|n+r;(\alpha,\beta))
\end{aligned}
\end{equation*}
for $r\in\mathbb{Z}_{\ge0}$. 
We remark that the relation (12) is a generalization of a relation for MZSVs in Ohno and Wakabayashi \cite[Examples (b)]{ow}, Zlobin \cite[Theorem 5]{zl2}. 
It is worth noting that the case $m=1$ of (12) can be derived from the hypergeometric identity of 
Krattenthaler and Rivoal \cite[Proposition 1 (ii)]{kr}, which is a limiting case 
of Andrews's basic hypergeometric identity \cite[Theorem 4]{an}: 
indeed, it can be obtained by taking $a=\alpha+\beta$, $c_0=\alpha$, $b_i=\alpha$, $c_i=\beta$ ($i=1,\ldots,s-1$), 
$b_s=\beta$, $c_s=1$ ($s\in\mathbb{Z}_{\ge2}$, $\alpha,\beta\in\mathbb{C}$ such that $\mathrm{Re}(\alpha)>0$, 
$\beta\notin\mathbb{Z}_{\le0}$) in \cite[Proposition 1 (ii)]{kr}. 
The same way as above can be applied to deriving relations among (1) and (2) 
from (3). 
Using 
\begin{equation*}
\frac{(\alpha)_{m_1}}{(\alpha)_{m_n}}
=
\prod_{i=2}^{n}\frac{(\alpha)_{m_{i-1}}}{(\alpha)_{m_i}}
\end{equation*}
and the same calculus as in the proof of (10), we have 
\begin{equation*}
\begin{aligned}
&\frac{(-1)^r}{r!}\frac{\mathrm{d}^r}{\mathrm{d}\alpha^r}
\left(\frac{(\alpha)_{m_1}}{(\alpha)_{m_n}}\right)\\
=&
\frac{(\alpha)_{m_1}}{(\alpha)_{m_n}}
\sum_{\begin{subarray}{c}
r_1+\cdots+r_{n-1}= r\\
       r_i\in\mathbb{Z}_{\ge0}
      \end{subarray}}
\sum_{\begin{subarray}{c}
m_1{\le}m_{11}\le\cdots{\le}m_{1 r_1}<m_2\\
\vdots\\
m_{n-1}{\le}m_{n-1 1}\le\cdots{\le}m_{n-1 r_{n-1}}<m_n
\end{subarray}}
\prod_{i=1}^{n-1}\prod_{j=1}^{r_i}\frac{1}{m_{ij}+\alpha}
\end{aligned}
\end{equation*}
for $n\in\mathbb{Z}_{\ge1}$, $r\in\mathbb{Z}_{\ge0}$ and $m_1,\ldots,m_n\in\mathbb{Z}$ such that $0\le{m_1}<\cdots<m_n$. 
This identity gives the following expression of the partial derivatives on $\alpha$ of (1): 
\begin{equation}
\begin{aligned}
&\frac{(-1)^r}{r!}\frac{\mathrm{\partial}^r}{\mathrm{\partial}\alpha^r}
\left(\sum_{0\le m_1< \cdots<m_{n}<\infty}
\frac{(\alpha)_{m_1}}{{m_1}!}
\frac{{m_n}!}{(\alpha)_{m_n}}
\left\{\prod_{i=1}^{n}\frac{1}{(m_i+\alpha)^{a_i}(m_i+\beta)^{b_i}}\right\}\right)\\
=
&\sum_{\begin{subarray}{c}
r_1+\cdots+r_{n-1}\\
+s_1+\cdots+s_n=r \\
r_i,s_i\in\mathbb{Z}_{\ge0}
      \end{subarray}}
\left\{ \prod_{i=1}^{n}\binom{a_i-1+s_i}{s_i}\right\}
Z_{(\{r_i\}_{i=1}^{n-1})}(\{a_i+s_i\}_{i=1}^{n}|
\{b_i\}_{i=1}^{n};(\alpha,\beta))
\end{aligned}
\end{equation}
for $r\in\mathbb{Z}_{\ge0}$, where 
\begin{equation}
\begin{aligned}
&Z_{(\{r_i\}_{i=1}^{n-1})}(\{a_i\}_{i=1}^{n}|\{b_i\}_{i=1}^{n};(\alpha,\beta))\\
:=&
\sum_{\begin{subarray}{c}
0=m_0{\le}m_1{\le}m_{11}\le\cdots{\le}m_{1 r_1}<m_2\\
\vdots\\
m_{n-1}{\le}m_{n-1 1}\le\cdots{\le}m_{n-1 r_{n-1}}<m_n<\infty
\end{subarray}}
\frac{(\alpha)_{m_1}}{{m_1}!}\frac{{m_n}!}{(\alpha)_{m_n}}\\
&\times\left\{\prod_{i=1}^{n-1}\frac{1}{(m_i+\alpha)^{a_i}(m_i+\beta)^{b_i}}
\left(\prod_{j=1}^{r_i}\frac{1}{m_{ij}+\alpha}\right)\right\}
\frac{1}{(m_n+\alpha)^{a_n}(m_n+\beta)^{b_n}},
\end{aligned}
\end{equation}
where $n\in\mathbb{Z}_{\ge1}$; $r_i\in\mathbb{Z}_{\ge0}$ ($i=1,\ldots,n-1$); 
the conditions for $a_i$, $b_i$, $\alpha$, $\beta$ are the same as those in (1). 
For $n=1$, this multiple series becomes
\begin{equation*}
Z_{(\emptyset)}(a_1|b_1;(\alpha,\beta))=\sum_{\begin{subarray}{c}0{\le}m_1<\infty\end{subarray}}
\frac{1}{(m_1+\alpha)^{a_1}(m_1+\beta)^{b_1}},
\end{equation*}
where $\emptyset$ is the empty set. (A characteristic of (14) is that 
its summation consists of both $<$ and $\le$. The usefulness of multiple series with such a summation was shown 
by Fischler and Rivoal \cite{fr}, Kawashima \cite{kaw}, Ulanskii \cite{ul}.) 
It is easy to verify that the multiple series (14) can be expressed 
in a $\mathbb{Z}$-linear combination of (1) or (2), and that 
the partial derivative on $\beta$ of (1) can be expressed in a 
$\mathbb{Z}$-linear combination of (1). Therefore differentiating both sides of (3) also gives relations among (1) and (2). For instance, we can obtain the following: 
Taking $k_i=2$ ($i=1,\ldots,n$) in (3), we have 
\begin{equation}
\begin{aligned}
Z_{I}(1,\{2\}^{n};(\alpha,\beta))
=
Z_{II}(\{2\}^{n-1},3;(\alpha,\beta))
\end{aligned}
\end{equation}
for $n\in\mathbb{Z}_{\ge1}$, $\alpha,\beta\in\mathbb{C}$ such that 
$\mathrm{Re}(\alpha)>0$, $\beta\notin\mathbb{Z}_{\le0}$. 
Differentiating both sides of (15) $r$ times on $\alpha$ and using (13), we have the relation 
\begin{equation}
\begin{aligned}
&\sum_{\begin{subarray}{c}
r_1+\cdots+r_{n}\\
+s_1+\cdots+s_n=r \\
r_i,s_i\in\mathbb{Z}_{\ge0}
      \end{subarray}}
Z_{(\{r_i\}_{i=1}^{n})}(0,\{1+s_i\}_{i=1}^{n}|\{1\}^{n+1};(\alpha,\beta))\\
=
&
\sum_{\begin{subarray}{c}
r_1+\cdots+r_{n-1}\\
+s_1+\cdots+s_n=r \\
r_i,s_i\in\mathbb{Z}_{\ge0}
      \end{subarray}}
(1+s_n)Z_{(\{r_i\}_{i=1}^{n-1})}(\{1+s_i\}_{i=1}^{n-1},2+s_n|\{1\}^{n};(\alpha,\beta))
\end{aligned}
\end{equation}
for $r\in\mathbb{Z}_{\ge0}$. On the other hand, differentiating both sides of (15) $r$ times on $\beta$, 
we have the relation 
\begin{equation}
\begin{aligned}
&\sum_{\begin{subarray}{c}
r_1+\cdots+r_{n+1}=r \\
r_i\in\mathbb{Z}_{\ge0}
      \end{subarray}}
Z_{I}(1+r_1,\{2+r_i\}_{i=2}^{n+1};(\alpha,\beta))\\
=
&\sum_{\begin{subarray}{c}
r_1+\cdots+r_{n}=r \\
r_i\in\mathbb{Z}_{\ge0}
      \end{subarray}}
Z_{II}(\{2+r_i\}_{i=1}^{n-1},3+r_n;(\alpha,\beta))
\end{aligned}
\end{equation}
for $r\in\mathbb{Z}_{\ge0}$. 
As pointed out in a preprint of the author distributed on February 12, 2015, 
the relation (15) is a generalization of the duality formula for $\zeta(1,\{2\}^{n})$. 
In fact, it is a new and non-trivial two-parameter generalization; compare (15) with \cite[Lemma 2.3]{i2009}. 
\begin{problems}[December 2018 and 2022] 
The following problems still defy a solution. I hope that 
investigating them exposes the reader to new viewpoints. 
The problems (\textit{i}), (\textit{ii}) and (\textit{iii}) below are on 
appropriate generalizations and the problem (\textit{iv}) on appropriate methods. 
\par 
(\textit{i}) It is interesting to generalize my duality formula (15) to a full duality formula 
 in appropriate ways and also to more general duality relations among (1). 
(The research on the duality of (1) has been part of my research 
project since 2007.) As a note on the latter part of this problem, 
I point out that my relation (17) is a two-parameter generalization of Ohno's relation for MZVs 
\cite[Theorem 1]{o1999} for the index $(1,\{2\}^{n})$ ($n\in\mathbb{Z}_{\ge1}$). 
Ohno's relation is a generalization of the duality formula for MZVs. 
(My relations (15), (16) and (17) were written in my preprint distributed on February 12, 2015.) 
\par (\textit{ii}) In \cite[Remark 7 (i)]{i4}, I gave a new proof of Hoffman's identity 
$\zeta(\{1\}^k, l+2)=\zeta(\{1\}^l, k+2)$ ($k,l\in\mathbb{Z}_{\ge0}$; 
\cite[Theorem 4.4]{h}), which is the duality formula for $\zeta(\{1\}^k,l+2)$. 
My proof is based on the hypergeometric identities \cite[Theorem 4]{an} and \cite[Proposition 1 (i)]{kr}; 
therefore it can be regarded as a hypergeometric proof of the duality formula. 
It is interesting to generalize my proof to a proof of the duality formula 
for all MZVs in appropriate ways. 
\par 
(\textit{iii}) It is interesting to generalize my results in \cite{i2011}, \cite{i4}, \cite{i2020} 
and to extend my ideas in there in appropriate ways, e.g., the identities (29), 
(30), (28), (58), (27), (57), (66), (67), (70), (71) of \cite{i4} and the ideas for their proofs, namely, 
my identities for 
\begin{equation*}
\begin{aligned}
&\zeta^{\star}(\{1\}^{k+1}, \{2\}^{l+1}),\,\, \zeta^{\star}(k+2,\{2\}^l), \,\,
\zeta(\{1\}^k, l+2), \,\,\zeta^{\star}(\{1\}^k, l+2),\\ 
&\zeta_{k+1,+}^{<}(\{1\}^k, l+2;\alpha), \,\,\zeta_{k+1,+}^{\le}(\{1\}^k, l+2;\alpha), 
\,\,\zeta_{k+1,+}^{<}(\{1\}^{k+1}|\{0\}^k, l+1;(\alpha,\beta)),\\
&\zeta_{k+1,+}^{\le}(\{1\}^{k+1}|\{0\}^k, l+1;(\alpha,\beta)) 
\quad (k,l\in\mathbb{Z}_{\ge0}; 
\alpha,\beta\in\mathbb{C}\setminus\mathbb{Z}_{\le0}) 
\end{aligned}
\end{equation*}
and my ideas for their proofs. (I proved the identities (28), (58), (27), (57) 
 in 2014 and the identities (66), (67), (70), (71) in January 2015. For (29) and (30), see \cite{i2011}.) 
\par 
(\textit{iv}) In \cite{i4}, I studied four-parameter multiple series, (1) and (41) of \cite{i4}. (The multiple series (1) of \cite{i4} is a four-parameter extension of 
both (1) and (2) of the present paper.) I studied relations among them 
in there (see also \cite{dual-3}). It is interesting to study analytic continuations of my four-parameter multiple series (1) and (41) with complex variables $a_i$, $b_i$, $c_i$, $d_i$, $s_i$, $t_i$ 
in appropriate ways. (This requires appropriate methods of handling 
the Pochhammer factors of (1) and (41).) 
This problem grew out of a reading of 
the papers Akiyama, Egami and Tanigawa \cite{aet} and 
Goncharov \cite[pp.~25--32]{gon}, where the authors proved analytic continuations of 
a multiple zeta-function and results on its values at non-positive integers. 
The problem grafts a problem on my research objects and on the above authors's. 
\end{problems}
\begin{note}[2022] 
(\textit{i}) I note that a revised manuscript of \cite{i} which contains Remark 2.7 was first submitted to the journal on August 4, 2009. 
This is an additional explanation for the addition of Remark 2.7 to \cite{i}. 
See also an explanation at the beginning of \cite[Remark 2.7]{i} and 
my preprint arXiv:0908.2536v1. 
\par 
(\textit{ii}) In \cite[Remark 2.7]{i}, I noted a connection between the hypergeometric identity 
\cite[Proposition 1 (ii)]{kr} and a multiple zeta-star object (the case $\alpha=\beta$ of my multiple series (2)). The connection was noted also in 
my preprint arXiv:0908.2536v1, posted on August 18, 2009. 
It appears that, after my note, some researchers took an interest in 
the hypergeometric identities \cite[Theorem 4]{an}, \cite[Proposition 1]{kr} 
and their proofs. In 2009, based on the connection, I began an application of 
the above hypergeometric identities to the study of relations for MZVs, MZSVs 
and special values of multiple polylogarithms. 
After that, I obtained many results; see \cite{i2011}, \cite{i4}, \cite{i2020}. 
The identities (29), (30) and (28), (58) of \cite{i4}, among others, have been known to be results of mine since 2011 and 2014, respectively (see also \cite{i2011}). 
(The first two identities show the rightness of my observations stated in \cite[(\textbf{A3}) and (\textbf{A4})]{i2011}; see \cite[Addendum]{i2011}.) 
I obtained most of the results of \cite{i4} and Applications 1 and 2 of \cite{i2020} in 2013--2015 (see \cite[Notes 1 and 2]{i2020}). I note that most of the contents of \cite{i4} and Application 1 of \cite{i2020} were written in preprints of mine distributed in 2013 and 2015. 
These preprints were sent to a Japanese zeta-researcher (a JZ for short). 
To be more precise, an earlier version of \cite{i4} and its revised version 
were sent by e-mail to the JZ in October and December 2013, respectively, 
and further, a preprint of \cite{i4} was also sent by e-mail 
to the JZ in February 2015. (Application 1 of \cite{i2020} was written in 
all these preprints.) These pieces of information about my research 
activities can be confirmed by asking the JZ. 
In the preprint sent in February 2015, I wrote the identities (28) and (58) of \cite{i4}, which are new identities for $\zeta(\{1\}^k, l+2)$ 
and $\zeta^{\star}(\{1\}^k, l+2)$. It was already known in 2014 that 
I discovered these new identities, 
because information on my discovery was $\ast\ast\ast$ by $\ast\ast\ast$ 
to $\ast\ast\ast$ in 2014. It appears that my identities are revelations about 
MZVs and MZSVs. Hence it is interesting to generalize my identities 
in appropriate ways. 
\end{note} 
\par
$\textbf{(C2)}$ 
For $k_1,\ldots,k_n\in\mathbb{Z}_{\ge1}$ such that $k_1+\cdots+k_n=m+n$ 
($m,n\in\mathbb{Z}_{\ge1}$), the right-hand side of (4) becomes
\begin{equation*}
mZ(n|m+1;(\alpha,\beta))+nZ(n+1|m;(\alpha,\beta)). 
\end{equation*} 
This sum is invariant under the replacement 
$(\alpha, \beta, m, n){\leftrightarrow}(\beta, \alpha, n, m)$; 
therefore the left-hand side of (4) also has the following symmetry: 
\begin{equation*}
\begin{aligned}
&\sum_{i=1}^{n}\sum_{j=0}^{k_i-2}Z^{\star}_{I}(j+1,k_{i+1},\ldots,k_n,k_1,\ldots,k_{i-1},k_i-j;(\alpha,\beta))\\
=
&\sum_{i=1}^{m}\sum_{j=0}^{l_i-2}Z^{\star}_{I}(j+1,l_{i+1},\ldots,l_m,l_1,\ldots,l_{i-1},l_i-j;(\beta,\alpha))
\end{aligned}
\end{equation*}
for any fixed $m,n\in\mathbb{Z}_{\ge1}$, all $k_1,\ldots,k_n, l_1,\ldots,l_m\in\mathbb{Z}_{\ge1}$ such that 
$k_1+\cdots+k_n=l_1+\cdots+l_m=m+n$ and $\alpha,\beta\in\mathbb{C}$ with 
$\mathrm{Re}(\alpha),\mathrm{Re}(\beta)>0$. 
An application of this kind of symmetry of parametrized multiple series 
was studied also in \cite{i2009}.
\par
$\textbf{(C3)}$ It is interesting to investigate whether CSFs for MZVs and MZSVs can be generalized to three- or more parameter multiple series. 
In fact, the results will have interesting applications to the study of relations among MZVs and MZSVs. For example, one will be able to derive relations among them from the results by using our method used in $\textbf{(C1)}$ and \cite{i4}. (Our method was developed in preprints of \cite{i4} distributed in 2013.) 
We shall pose a problem on this topic. 
We mainly consider this topic for our three-parameter multiple series 
\begin{equation}
\begin{aligned}
&\sum_{0{\le}m_1{\le}\cdots{\le}m_n<\infty}
\frac{(\alpha)_{m_1}}{{m_1}!}\frac{(\beta)_{m_1}}{(\gamma)_{m_1}}
\frac{{m_n}!}{(\alpha)_{m_n}}\frac{(\gamma)_{m_{n}}}{(\beta)_{m_n}}\\
&\times
\left\{\prod_{i=1}^{n}\frac{1}{(m_i+\alpha)^{a_i}(m_i+\beta)^{b_i}(m_i+\gamma)^{c_i}}\right\},
\end{aligned}
\end{equation}
where $n\in\mathbb{Z}_{\ge1}$; $a_{i},b_{i},c_i\in\mathbb{Z}$ ($i=1,\ldots,n$) such that $a_i+b_i+c_i\ge1$ ($i=1,\ldots,n-1$), 
$a_{n}+b_{n}+c_{n}\ge2$; $\alpha,\beta,\gamma\in\mathbb{C}\setminus\mathbb{Z}_{\le0}$ such that $\mathrm{Re}(\alpha+\beta-\gamma)>0$: these conditions guarantee the absolute convergence of (18) 
(see \cite[(3.12)]{kr}). The multiple series (18) is an extension of (2). 
The multiple series (18) was discovered by the author in 2014 (before September 17, 2014) by studying the hypergeometric identities of Andrews \cite[Theorem 4]{an}, Krattenthaler and Rivoal \cite[Proposition 1]{kr} (see \cite[Note 2 (iv)]{i2020} and \cite[Introduction]{i4}). Our previous remark \cite[Remark 2.7]{i} played an essential role for discovering (18) and its identities. Here we put 
\begin{equation}
\begin{aligned}
&Z^{\star}_{I}(\{k_i\}^{n}_{i=1};(\alpha, \beta,\gamma))\\
:=
&\sum_{0\le m_1{\le} \cdots {\le}m_n<\infty}\frac{(\alpha)_{m_1}}{{m_1}!}\frac{(\beta)_{m_1}}{(\gamma)_{m_1}}
\frac{{m_n}!}{(\alpha)_{m_n}}\frac{(\gamma)_{m_{n}}}{(\beta)_{m_n}}\\
&\times\frac{1}{(m_1+\gamma)^{k_1}}
\left\{\prod_{i=2}^{n}\frac{1}{(m_i+\alpha)(m_i+\beta)(m_i+\gamma)^{k_i-2}}\right\}
\end{aligned}
\end{equation}
and 
\begin{equation*}
Z(a|b|c;(\alpha^{'},\beta^{'},\gamma^{'})):=\sum_{m=0}^{\infty}
\frac{1}{(m+\alpha^{'})^{a}(m+\beta^{'})^{b}(m+\gamma^{'})^{c}},
\end{equation*}
where $k_i\in\mathbb{Z}_{\ge1}$ ($i=1,\ldots,n-1$), $k_n\in\mathbb{Z}_{\ge2}$; 
$a,b,c\in\mathbb{Z}$ such that $a+b+c\ge2$; $\alpha,\beta,\gamma\in\mathbb{C}\setminus\mathbb{Z}_{\le0}$ such that $\mathrm{Re}(\alpha+\beta-\gamma)>0$; $\alpha^{'},\beta^{'},\gamma^{'}\in\mathbb{C}\setminus\mathbb{Z}_{\le0}$. The above single series is the case $n=1$ of (18). The multiple series (19) is an extension of both the multiple series $Z^{\star}(\{k_i\}^{n}_{i=1};\alpha)$ of \cite{i} and $Z^{\star}_{I}(\{k_i\}^{n}_{i=1};(\alpha,\beta))$, which are the cases $\alpha=\beta=\gamma$ and $\beta=\gamma$ of (19), respectively. CSF for MZSVs can perhaps be regarded as a generalization of the relation
\begin{equation}
\zeta^{\star}(1,\{2\}^{s-1})=2\zeta(2s-1)
\end{equation}
for $s\in\mathbb{Z}_{\ge2}$ (see Ohno and Zudilin \cite[p.~326]{oz}, 
Ohno and Wakabayashi \cite{ow}). Indeed, for the index $(1,\{2\}^{s-1})$, 
it becomes (20). 
We observed that the same situation occurs also for 
the one-parameter multiple series $Z^{\star}(\{k_i\}^{n}_{i=1};\alpha)$ of \cite{i} and the two-parameter multiple series (2), that is, each of these multiple series satisfies the same relation as (20); see \cite[Remark 2.7]{i} and 
a note on the case $m=1$ of (12) stated in \textbf{(C1)}. Furthermore, each relation can be generalized to CSF for the multiple series; see \cite[Theorem 1.1 (ii)]{i} and Theorem 1. 
We further studied whether this situation occurs also for more multi-parameter multiple series. As a result, we found that the three-parameter multiple series (19) also satisfies the same relation as (20), 
\begin{equation}
Z_{I}^{\star}(1,\{2\}^{s-1};(\alpha, \beta,\gamma))
=Z(s-1|s-1|1;(\alpha,\beta,\gamma))+Z(s-1|s-1|1;(\alpha,\beta,\alpha+\beta-\gamma))
\end{equation}
for $s\in\mathbb{Z}_{\ge2}$, $\alpha,\beta,\gamma\in\mathbb{C}\setminus\mathbb{Z}_{\le0}$ such that $\mathrm{Re}(\alpha+\beta-\gamma)>0$. 
(For details, see a note at the end of the present part.) 
This relation gives a generalization of the following three relations: (20), the case $m=1$ of (12), and the relation in \cite[Remark 2.7]{i}, which are the cases 
$\alpha=\beta=\gamma=1$, $\beta=\gamma$, and $\alpha=\beta=\gamma$ of (21), respectively. 
On the basis of the above facts, we pose the problem whether the relation (21) can be generalized to CSF for the three-parameter multiple series (18). (This problem is what the author wrote in research documents of February 2016.) We note that 
the relation (21) has an interesting application to the study of relations among MZVs and MZSVs. Indeed, we showed in \cite[$\textbf{(R2)}$ of Section 3]{i4} that the sum formulas for MZVs and MZSVs can be derived from (21) by partial differentiation. This application gives a new proof of the sum formulas, which is hypergeometric. From these facts, we think that the results of the above problem also have this kind of interesting application. 
Finally, we give some notes on (18) and (21). 
Not only (18) but also the four-parameter multiple series of 
\cite{i4} was discovered by the author in 2014 (before September 17, 2014) by studying the hypergeometric identities of Andrews \cite[Theorem 4]{an}, Krattenthaler and Rivoal \cite[Proposition 1]{kr} (see \cite[Note 2 (iv)]{i2020} and \cite[Introduction]{i4}). The results of \cite{i4} already show the potentiality of these multiple series (see also \cite[Note 2]{i2020}). Our previous remark \cite[Remark 2.7]{i} played an essential role for discovering these multiple series and their identities. 
The multiple series (18) converges uniformly in any compact 
subset of the domain $\{(\alpha,\beta,\gamma)\in\mathbb{C}^3$ : $\mathrm{Re}(\alpha),\mathrm{Re}(\beta)>0$, $\mathrm{Re}(\alpha+\beta)>\mathrm{Re}(\gamma)>0\}$; therefore it is holomorphic in this domain as a function of $(\alpha, \beta, \gamma)$. 
The relation (21) can be derived from the hypergeometric identity 
of \cite[Proposition 1 (ii)]{kr} by taking $a=\alpha+\beta$, $c_0=\alpha+\beta-\gamma$, $b_i=\alpha$, $c_i=\beta$ ($i=1,\ldots,s-1$), 
$b_s=1$, $c_s=\gamma$ ($s\in\mathbb{Z}_{\ge2}$, $\alpha,\beta,\gamma\in\mathbb{C}\setminus\mathbb{Z}_{\le0}$ such that 
$\mathrm{Re}(\alpha+\beta-\gamma)>0$). 
For another generalization of (20) and a related problem, 
see \cite[Corollary 2.3 (i) and $\textbf{(R1)}$ of Section 3]{i4}. 
\begin{note}[2022] 
I give some notes related to my papers \cite{i2007} and \cite{i2009}. 
The study of the multiple series (1) was originated in my master's thesis \cite{i2007}. I obtained the results of \cite{i2007} in 2006. 
One of the main results is the sum formula for the case $\alpha=\beta$ of my multiple series (1). 
In 2006, I talked on the sum formula and its proof at my then academic adviser's seminar. On February 3, 2007, my master's thesis passed the final examination. 
From some facts, I concluded in February 2007 that a certain person 
had already been aware of the contents of my master's thesis \cite{i2007} at 
the workshop ``Zeta Wakate Kenky\={u}sh\={u}kai" held at Graduate School of Mathematics, Nagoya University, Japan, February 17--18, 2007. 
(I am writing about this matter in \cite{add1}.) 
After my master's thesis, I generalized my sum formula 
mentioned above to a large class of relations; see \cite[Theorem 1.1]{i2009}. 
I talked on my research \cite{i2009} and related notes 
(see \cite[p.~29]{dual-2}) 
at Seminar on Analytic Number Theory, Graduate School of Mathematics, Nagoya University, Japan, 
February 13, 2008. See also \cite[Acknowledgments on p.~578]{i2009}. 
(The following people were parts of the audience of my talk: 
Kohji Matsumoto, Yoshio Tanigawa, Takashi Nakamura, Yoshitaka Sasaki.) 
The title of my talk was ``On Ohno's relation for certain multiple series". 
This title was announced both in Japanese and English on the website of 
Graduate School of Mathematics, Nagoya University, Japan about one week before. In addition, the title was also announced on February 5, 2008 
by using a mailing list of the Seminar. 
My multiple series (1) and (2) and their extensions studied in \cite{i4}, \cite{dual-2} 
are highly non-trivial objects. I am writing an expository paper on my previous 
papers on these multiple series, which is \cite{add3}. 
\end{note}
\begin{corrections}
(i) Page 514, lines 11--13 from the bottom: Change ``$0<\mathrm{Re}\,\alpha<s-1$" 
into ``$\mathrm{Re}\,\alpha>0$" and remove the sentence 
``The condition $0<\mathrm{Re}\,\alpha<s-1$ can be \ldots as functions of $\alpha$.". (ii) Page 517, line 12 from the bottom: ``March 2007" should be 
``February 3, 2007". 
\end{corrections}
\begin{flushleft}
Nagoya, Japan\\
\textit{E-mail address}: masahiro.igarashi2018@gmail.com
\end{flushleft}
\end{document}